\documentclass[12pt, a4paper]{amsart}

\usepackage{amsmath, amsfonts, amsthm, amssymb}

\newif\ifpdfAuthoring
\pdfAuthoringfalse
\ifpdfAuthoring
\input{hyperef.tex}
\fi

\newcommand     {\ab}[1][\SL]   {{#1}^{{\rm ab}}}
\newcommand     {\ach}[1][\D]   {\SL[{#1}]^{\sharp}}
\newcommand     {\card}[1]      {\sym{card}(#1)}
\newcommand     {\D}            {{\mathfrak o}}
\newcommand     {\F}            {{\mathbb F}}
\newcommand     {\f}[1][A]      {f_{#1}}
\newcommand     {\gs}[2][A]     {\{#1\}_{#2}}
\newcommand     {\id}[1]        {\mathfrak #1}
\newcommand     {\isom}         {\cong}
\newcommand     {\mat}[4]       {\big[\begin{smallmatrix}#1&#2\\#3&#4\end{smallmatrix}\big]}
\newcommand     {\Q}            {{\mathbb Q}}
\newcommand     {\rem}[1]       {\relax}
\newcommand     {\s}[1][A]      {s({#1})}
\newcommand     {\SL}[1][\D]    {\sym{SL}(2,{#1})}
\newcommand     {\sym}[1]       {\operatorname{#1}}
\newcommand     {\tr}           {\sym{t}}
\newcommand     {\Z}            {{\mathbb Z}}

\theoremstyle{plain}
\newtheorem{Theorem}{Theorem}
\newtheorem{Proposition}{Proposition}
\newtheorem{Corollary}{Corollary}
\newtheorem{Lemma}{Lemma}
\theoremstyle{definition}

\title{%
  Linear characters of $\sym{SL}_2$ over Dedekind domains}

\author{%
  Hatice Boylan and Nils-Peter Skoruppa}

\address{%
  RWTH Aachen and Universit\"at Siegen\endgraf
  hatice.boylan@gmail.com\endgraf
  \null
  Universit\"at Siegen\endgraf
  nils.skoruppa@gmail.com}

\thanks{%
  Parts of this article were prepared during the participation
  of the authors in the {\em Arithmetic Statistics} program (2011) at
  the {\em Mathematical Sciences Research Institute, Berkeley}, which
  was supported by the National Science Foundation of USA (Award
  Number 0757627).}

\subjclass[2010]{%
  Primary 20H05, Secondary 11F41}

\begin{document}

\begin{abstract}
  For an important class of arithmetic Dedekind domains~$\D$ including
  the ring of integers of not totally complex number fields, we
  describe explicitly the group of linear characters of~$\SL$. For
  this, we determine, for arbitrary Dedekind domains~$\D$, the group
  of linear characters of~$\SL$ whose kernel is a congruence subgroup.
\end{abstract}

\maketitle

\section{Statement of results and discussion}
\label{sec:discussion}

It is well-known that the group of linear characters of $\SL[\Z]$ is
cyclic of order~$12$.  The literature contains various formulas for
the linear characters of $\SL[\Z]$. These are either produced as a
byproduct in the theory of modular forms as consequence of the
transformation law of the Dedekind $\eta$-function under
$\SL[\Z]$~\cite{Dedekind1}, \cite{Dedekind2}, or else show up in the
purely group theoretical calculation of character tables of the groups
$G_N:=\SL[\Z/N\Z]$\rem{~cite{Frobenius}, cite{???}}. These groups
contain nontrivial linear characters only for $N=2,3,4$.  For these
$N$, the groups of linear characters of $G_N$ are generated by the
characters~$\varepsilon_N$, respectively, which are given by the
formulas
\begin{equation}
  \label{eq:explicit-formulas}
  \begin{split}
    \varepsilon_2(A) &= e^{\frac{2\pi i}2 \left(\tr(A)+1\right)\,\gs2}
    ,
    \\
    \varepsilon_3(A) &= e^{\frac{2\pi i}3 \tr(A)\,\gs3} ,
    \\
    \varepsilon_4(A) &= \s \, e^{\frac{2\pi i}4
      \left(\tr(A)+1\right)\,\gs 4} .
  \end{split}
\end{equation}
Here $\tr(A)$ denotes the trace of $A$. Moreover, for $A=\mat *bc*$ in
$G_N$, we use $\gs N = 0$ if $A=\pm 1$ or $N=4$ and $A\equiv 1\bmod
2$, we use $\gs N = 1$ if $\tr(A)^2\not=4$, and otherwise we have $\gs
N=u$, where $u=c$ if $c$ is a unit, and $u=-b$ otherwise. Finally,
$\s=-1$ if $A$ can be written in the form $1+2\mat abc*$ such that
$a+b+c$ is odd, and $\s=1$ otherwise. It is easy to see that the
functions $\varepsilon_N$ are class functions (see the remark at the
end of Section~\ref{sec:SL2Z}), but it is more difficult to see that
they are indeed linear characters. We were not able to find any
reference which describes the characters of $G_N$ by simple formulas
similar to~\eqref{eq:explicit-formulas}. We refer to
Section~\ref{sec:SL2Z} for a deduction of these formulas. Note that
$\varepsilon_N$ has order $N$. The group of linear characters of
$\SL[\Z]$ is generated by the two characters which are obtained by
composing $\varepsilon_3$ and $ \varepsilon_4$ with the natural
reduction maps from $\SL[\Z]$ onto $G_3$ and of $G_4$, respectively.

Using the same procedure we can define characters for $\SL$ for an
arbitrary ring~$\D$.  Namely, if $\id l$ is an ideal such that $\D/\id
l \isom \Z/N\Z$ for some $N\in\{2,3,4\}$, we have available the
nontrivial character
\begin{equation*}
  \varepsilon_{\id
    l} := \varepsilon_N \circ \sym{red}_{\id l}
\end{equation*}
of $\SL$, where $\sym{red}_{\id l}$ is the natural map from $\SL$ to
$G_N$ which maps a matrix $A$ to the matrix which is obtained by
replacing each entry $x$ of $A$ by the residue class modulo $N$ of any
integer $y$ such that~$x\equiv y \bmod \id l$.

By the very construction, $\varepsilon_{\id l}$ is trivial on matrices
$A$ which are congruent to $1$ modulo $\id l$.  Recall that a subgroup
of $\SL$ is called {\em congruence subgroup} if it contains the kernel
$\Gamma(\id l)$ of the natural map $\SL\rightarrow \SL[\D/\id l]$ for
some nonzero ideal $\id l$.  For an integral domain~$\D$, the
intersection of two congruence subgroups is again a congruence
subgroup since the intersection of $\Gamma(\id l)$ and $\Gamma(\id m)$
contains $\Gamma(\id l \id m)$, and $\id l \id m$ is nonzero if $\id
l$ and $\id m$ are so. Accordingly, for an integral domain~$\D$, the
linear characters of $\SL$ whose kernel is a congruence subgroup form
a subgroup of the group of all linear characters of $\SL$, which we
shall denote by~$\ach$.

As we shall show in Section~\ref{sec:proof}, the construction in the
penultimate paragraph is, for Dedekind domains $\D$, the only way to
obtain linear characters whose kernel is a congruence
subgroup. However, there is another character, which we have to
consider.  The prime ideals of Dedekind domains~$\D$ which have
residue field $\F_2$\footnote{We use $\F_p$ for the finite field with
  $p$ elements.}  fall into two classes, namely, those prime ideals
$\id q$ such that $\id q$ divides $2$ but $\id q^2$ does not, and
those prime ideals $\id r$ such that $\id r^2$ divides $2$. In the
former case $\D/\id q^2 \isom \Z/4\Z$, whereas in the latter case
$\D/\id r^2 \isom R:= \F_2[t]/(t^2)$. Note that we can write any
element $A$ in $\SL[R]$ uniquely in the form $A=A_0(1+\alpha B)$,
where $A_0$ is in $\SL[\F_2]$ and $B$ an element of the additive group
of matrices over $\F_2$ whose traces equal zero, and where we use
$\alpha = t+(t^2)$. We define a linear character $\varepsilon_4'$ on
$\SL[R]$ by setting
\begin{equation}
  \label{eq:exceptionel-char}
  \varepsilon_4'(A)
  =(-1)^{a+b+c}
  \qquad(A=A_0\big(1+\alpha\mat abc*\big))
  .
\end{equation}
(For the proof that this defines indeed a character see
Section~\ref{sec:SL2Z}.) Again, if $\id l$ is an ideal of the ring
$\D$ such that $\D/{\id l}\isom R$ we set
\begin{equation*}
  \varepsilon_{\id l}'
  :=
  \varepsilon_4'\circ \sym{red}_{\id l}
  ,
\end{equation*}
where $\sym{red}_{\id l}$ denotes the natural map from $\SL$ to
$\SL[R]$ which is obtained by reducing the entries of a matrix over
$\D$ modulo~$\id l$ and applying the unique isomorphism from $\D/\id
l$ onto $R$ (the uniqueness of such an isomorphism follows form the
fact that $R$ has no nontrivial automorphisms).

\begin{Theorem}
  \label{thm:main}
  Let $\D$ be a Dedekind domain\footnote{In this note fields are not
    considered as Dedekind domains.}. Then the group $\ach$ of linear
  characters of $\SL$ whose kernel is a congruence subgroup equals the
  direct product
  \begin{equation}
    \label{eq:the-characters}
    \prod_{\id p} \langle\varepsilon_{\id p}\rangle
    \times \prod_{\id q\Vert2} \langle\varepsilon_{\id q^2}\rangle
    \times \prod_{\id r^2|2}
    \big(\langle\varepsilon_{\id r}\rangle \times \langle\varepsilon_{\id r^2}'\rangle\big)
    ,
  \end{equation}
  where $\id p$, $\id q$ and $\id r$ run through all prime ideals of
  $\D$ such that $\D/\id p=\F_3$, $\D/\id q=\F_2$, $\D/\id r=\F_2$,
  and such that $\id q^2$ does not divide $2$ and $\id r^2$
  divides~$2$.
\end{Theorem}

The assumption of the theorem excludes fields. However, if~$\D$ is a
field then $\SL$ does anyway not possess nontrivial linear characters
unless $\D=\F_2$ or $\D=\F_3$ (see the remark at the end of
Section~\ref{sec:proof}).

For a Dedekind domain $\D$, the number of ideals $\id p$ such that
$\D/\id p=\F_p$, for a given prime number $p$, is always finite (since
every such ideal divides the ideal generated by all elements of the
form $a^p-a$, which is not zero if $\D$ is infinite). Thus, for a
Dedekind domain $\D$, the group $\ach$ is always finite.  This holds
not true for an arbitrary integral domain. The ring $\F_2^\Z$ of all
maps from $\Z$ to $\F_2$ with argumentwise addition and multiplication
provides a counterexample. Here, for any integer $n$, the map
$A\mapsto \varepsilon_{3\Z}\big(A(n)\big)$ yields an element
of~$\ach[{\F_2^\Z}]$.

There is another interesting consequence of
Theorem~\ref{thm:main}. Namely, if, for an arbitrary ring~$\D$, a
subgroup $\Gamma$ of $\SL$ contains the commutator subgroup of $\SL$
then, for every $A$ which is not in $\Gamma$ there exists a linear
character $\chi$ whose kernel is $\Gamma$ but such that
$\chi(A)\not=1$\footnote{Since~$\Gamma$ contains the commutator
  subgroup, it is normal.  Any maximal extension of a linear character
  $\psi$ on $\langle A\Gamma\rangle$ with $\psi(A\Gamma)\not=1$ to a
  subgroup of $G:=\SL/\Gamma$, whose existence is ensured by Zorn's
  lemma, has by a standard argument $G$ as its domain, and therefore
  provides such a character.}. If $\D$ is a Dedekind domain, and if
$\Gamma$ is a congruence subgroup, then $\chi$ equals one of the
characters in $\ach$. Accordingly, $A$ is not contained in the
intersection $B(\D)$ of the kernels of the linear characters in
$\ach$. In other words, $\Gamma$ contains $B(\D)$. But $B(\D)$ is of
finite index in $\SL$. Hence we have
\begin{Corollary}
  Let $\D$ be a Dedekind domain. The number of congruence subgroups of
  $\SL$ containing the commutator subgroup $K(\D)$ is finite.
\end{Corollary}

In particular, we see that the commutator subgroup of $\SL$, for any
Dedekind domain $\D$, possesses a congruence closure, namely the
subgroup $B(\D)$. This is surprising since the index of the commutator
subgroup in $\SL$ is not necessarily finite (see below).

If $K$ is a global field, i.e.~an algebraic number field or a function
field in one variable over a finite field, then, for every finite set
of places of $K$ comprising the archimedean ones, the ring $\D_S$ of
$S$-integers is a Dedekind domain. The ring $\D_S$ consists of all
elements of $K$ whose valuation at a place $v$ of $K$ is $\ge 0$
unless $v$ is in $S$. If $\card S\ge 2$ then the abelianization
$\ab[{\SL[\D_S]}]$ of $\SL[\D_S]$ is finite (\cite[Thm.~3,
Cor.]{Serre}).  Moreover, if at least one place of $S$ is real or
nonarchimedean then all subgroups of finite index in $\SL[\D_S]$ are
congruence subgroups (\cite[Thm.~2, Cor.~3]{Serre}). Hence, as a
consequence of Theorem~\ref{thm:main} we obtain
\begin{Theorem}
  \label{thm:main+Serre}
  Assume that $\D$ is the ring of $S$-integers of a global field,
  where $\card S\ge2$ and $S$ contains at least a real or a
  nonarchimedean place. Then the group of {\em all} linear characters
  of $\SL$ coincides with~$\ach$, i.e.~it equals the
  group~\eqref{eq:the-characters}.
\end{Theorem}

Note, that the theorem implies, for function fields of one variable
over a finite field $\F$ and every finite set $S$ of at least two
places, that $\SL[\D_S]$ has no nontrivial character if the
characteristic of~$\F$ is different from~$2$ and~$3$.

The assumptions of Theorem~\ref{thm:main+Serre} apply, in particular,
to the ring of integers of a number field different from $\Q$ which is
not totally complex.
\begin{Corollary}
  \label{cor:maximal-order-case}
  Let $\D$ be the ring of integers of an algebraic number field with
  at least one real embedding. Then the group of {\em all} linear
  characters of $\SL$ coincides with the group $\ach$, i.e.~it equals
  the group~\eqref{eq:the-characters}.  In particular, the order of
  the abelianized group $\ab$ equals $3^a4^b$, where $a$ and $b$
  denote the number of prime ideals of degree~$1$ over $3$ and over
  $2$, respectively.
\end{Corollary}

In Section~\ref{sec:statistics} we present some numerics concerning
the abelianized groups~$\SL$ for rings of integers of number fields.

For imaginary quadratic fields the situation can be very different.
Indeed, let~$\D$ equal the ring of integers in $\Q(\sqrt {-7})$.  The
abelianized group $\ab$ is isomorphic to $\Z\oplus\Z_4$ (see
~\cite[p.~162]{Cohn}), hence the group of linear characters equals
$\text{S}^1\times\Z_4$, where $\text{S}^1$ is the subgroup of complex
numbers of modulus~$1$. In particular, here the group $\SL$ possesses
infinitely many characters of finite order. This example shows that,
even for a Dedekind ring, the group $\ach$ can be smaller than the
torsion subgroup of the group of linear characters.

We do not know whether the corollary holds true in general for totally
complex number fields with at least $2$ complex places. We have many
examples where it holds true (see Section~\ref{sec:statistics}), but
do not know of any counterexample.

As already mentioned we did not find any reference for the
formulas~\eqref{eq:explicit-formulas}. Though one can verify them
easily by a case by case analysis of conjugacy classes in $G_N$ using
a character table (see e.g.~\cite{GAP4}), we decided to give a
deduction from scratch in Section~\ref{sec:SL2Z}. The proof of
Theorem~\ref{thm:main} is given in Section~\ref{sec:proof}.

\section{Proof of Theorem~\ref{thm:main}}
\label{sec:proof}

For an arbitrary ring $\D$, we use $\D^\ast$ for its group of
units. We set
\begin{equation*}
  T(a):=\mat 1a{}1\quad\text{($a$ in $\D$),}
  \quad
  E(a):=\mat a{}{}{a^{-1}}\quad\text{($a$ in $\D^\ast$),}
  \quad
  S = \mat {}{-1}1{}
  .
\end{equation*}
Recall that the matrices $T(a)$ and their transposes are called {\em
  elementary matrices}. Note also, that $\SL$ is generated by
elementary matrices if and only if it is generated by the matrices
$T(a)$ and $S$ (since $S=\mat 1{-1}{}1 \mat1{}11 \mat 1{-1}{}1$ and
$\mat 1{}a1 = ST(-a)S^{-1}$).

For a linear character $\chi$ of $\SL$, the map $a\mapsto
\chi\left(T(a)\right)$ defines a linear character $\xi$ of the
(additive) group $\D$.  The set of $a$ in $\D$ such that $\xi(ab)=1$
for all $b$ in $\D$ forms an $\D$-ideal $\id a$, which we call the
{\em annihilator of $\chi$}. For a nonzero ideal $\id l$, we say that
$\chi$ has {\em level $\id l$} if $\chi$ is trivial on the subgroup
$\Gamma(\id l)$ of all matrices in $\SL$ which are congruent
modulo~$\id l$ to the unit matrix. Clearly, if $\chi$ has level ${\id
  l}$ then ${\id l}$ contains the annihilator.

\begin{Lemma}
  \label{lem:annihilator}
  Let $\D$ be an arbitrary ring. Then the $\D$-ideal generated by the
  elements $u^2-1$, where $u$ runs through the group of units of $\D$,
  is contained in the annihilator of any linear character of $\SL$.
\end{Lemma}
\begin{proof}
  The lemma is an immediate consequence of the formula
  \begin{equation*}
    \label{eq:the-key}
    T(b{u}^2)=E(u)T(b)E(1/u)
    ,
  \end{equation*}
  valid for all $b$ in $\D$ and all units $u$.
\end{proof}

\begin{Lemma}
  \label{lem:the-key}
  Assume that $\SL$ is generated by elementary matrices.  Let $K$
  denote the commutator subgroup of $\SL$.  Then the application
  \begin{equation}
    \label{eq:the-key-map}
    \D \rightarrow \ab[{\SL}]:=\SL/K
    \quad
    a\mapsto T(a)K
  \end{equation}
  defines a surjective group homomorphism.
\end{Lemma}
\begin{proof}
  The lemma follows from the fact that $\big(ST(1)\big)^3=S^2$, so
  that $SK=T(1)^{-3}K$, from which we recognize that $\ab$ is
  generated by the elements $T(a)K$.
\end{proof}

\begin{Lemma}
  \label{lem:unit-reduction}
  Let $n\ge1$ and let $\id p^n$ be a prime ideal power in the Dedekind
  domain $\D$.  Suppose that $u^2=1$ for every $u$ in $\left(\D/\id
    p^n\right)^\ast$.  Then one of the following statements holds
  true:
  \begin{enumerate}
  \item $n=1$ and $\D/\id p^n = \F_2$ or $\D/\id p^n = \F_3$.
  \item $n=2$ or $n=3$, $\id p$ divides $2$ but $\id p^2$ does not,
    and $\D/\id p^n=\Z/2^n\Z$.
  \item $n=2$, $\id p^2$ divides $2$, and $\D/\id p^n=\F_2[t]/(t^2)$.
  \end{enumerate}
\end{Lemma}
\begin{proof}
  By assumption $U:=(\D/\id p)^\ast$ has exponent $\le2$. Since $\id
  p$ is maximal, $F:=\D/\id p$ is a field. Hence every finite subgroup
  of $U$ is cyclic. It follows that $U$ has order $1$, and hence
  $F=\F_2$, or else $U$ has order $2$, and hence $F=\F_3$.

  Suppose $n\ge 2$. Then $\D/\id p=\F_2$, since in the case of $\D/\id
  p=\F_3$ the group of units of $\D/\id p^2$ has order $6$. If $\id
  p^2$ does not divide $2$, then $\D/\id p^n=\Z/2^n\Z$, and since
  $\Z/16\Z$ has elementary divisors $2$ and~$4$, we conclude $n\le 3$.

  Suppose finally that $\id p^2$ divides $2$. Let $\alpha$ denote any
  element in $\id p$ which is not contained in $\id p^2$. Then
  $(1+\alpha)^2 = 1+2\alpha +\alpha^2 \not\equiv 1 \bmod \id p^3$,
  from which we conclude $n=2$. But $\D/\id p^2 = \F_2[\alpha+\id
  p^2]=\F_2[t]/(t^2)$.
\end{proof}

\begin{proof}[Proof of Theorem~\ref{thm:main}]
  For proving Theorem~\ref{thm:main} recall that $\D$ is now a
  Dedekind domain. Suppose that $\chi$ is a nontrivial linear
  character of $\SL$ whose kernel contains $\Gamma(\id l)$ for some
  nonzero ideal~$\id l$.  The homomorphism $\SL\rightarrow \SL[\D/\id
  l]$ induced by the canonical map $\D\rightarrow \D/\id l$ is
  surjective (\cite[cor.~5.2]{Bass}). We may therefore consider $\chi$
  as a character of $\SL[\D/\id l]$.  The canonical map
  \begin{equation*}
    \SL[\D/\id l]
    \rightarrow
    \prod_{\id p^l\Vert \id l} \SL[\D/\id p^l]
  \end{equation*}
  induced from the Chinese remainder theorem, where $\id p^l$ runs
  through the prime ideal powers dividing $\id l$, is an isomorphism
  of groups. Accordingly $\chi$ factors into a product of characters
  $\chi_{\id p}$, where $\chi_{\id p}$ has level~$\id p^l$. For the
  proof of Theorem~\ref{thm:main} we may therefore assume that~$\chi$
  is a nontrivial character of $\SL$ with level $\id l=\id p^l$ for
  some prime ideal ${\id p}$. It is clear that the annihilator~$\id a$
  of $\chi$ contains $\id l$, hence is of the form $\id p^n$. We claim
  that the group of units $(\D/\id a)^\ast$ has exponent $\le 2$.

  Indeed, let $u$ be a unit of $\D/\id a$.  The canonical map $x+\id l
  \mapsto x+\id a$ from $\D/\id l$ to $\D/\id a$ induces a surjective
  homomorphism of the group of units (since the group of units is
  formed by the residue classes which are relatively prime to $\id
  p$).  We therefore find a preimage $u'$ of $u$ in $\big(\D/\id
  l\big)^\ast$. But then, by Lemma~\ref{lem:annihilator}, $u'^2-1$ is
  contained in the annihilator of the character $\widetilde\chi(a+{\id
    l})= \chi(a)$ of $\SL[\D/\id l]$, which equals $\id a/\id l$. We
  conclude $u^2=1$.

  Since $\chi$ is nontrivial, the annihilator $\id a = \id p^n$ is
  different from $\D$, i.e.~we have $n\ge 1$.  Indeed, if the
  annihilator of $\chi$ equaled $\D$, the annihilator of
  $\widetilde\chi$ would be~$\D/\id l$. But Lemma~\ref{lem:the-key},
  applied to $\D/\id l$ instead of $\D$, would imply that
  $\widetilde\chi$ is trivial. For applying Lemma~\ref{lem:the-key} we
  need that $\SL[\D/\id l]$ is generated by elementary matrices, which
  is, in fact, well-known~\cite[Ch.~5, Cor.~9.3]{Bass-K-theory}.
  However, for a quotient $ \D/\id l$ of a Dedekind domain by a
  nonzero ideal $\id l$ this can more easily be proven directly as
  follows. By using the Chinese remainder theorem, we may assume that
  $\id l$ equals a prime ideal power~$\id p^l$. But then $\D/\id l$ is
  a local ring. Hence, if $A=\mat abcd$ is a matrix in $\SL[\D/\id
  l]$, then either $c$ is a unit (and then $A=T(a/c)ST(dc)E(c)$), or
  $a$ is unit (and then $A=ST(-c/a)ST(ba)E(-a)$).  For a unit $u$, we
  have $E(-u)=ST(1/u)ST(u)ST(1/u)$.

  We can therefore apply Lemma~\ref{lem:unit-reduction} to deduce that
  one of the three cases~(1) to~(3) apply to $\D/\id a=\D/\id
  p^n$. But this suffices to compute $\ab[{\SL[\D/\id p^l]}]$. Namely,
  since ${\SL[\D/\id p^l]}$ admits the nontrivial character
  $\widetilde \chi$, it possesses a nontrivial character whose
  annihilator is minimal. Without loss of generality we can assume
  that $\widetilde \chi$ assumes minimal annihilator. The
  map~\eqref{eq:the-key-map} of Lemma~\ref{lem:the-key} factors then
  through a surjection
  \begin{equation*}
    \D/\id p^n
    \rightarrow
    \ab[{\SL[\D/\id p^l]}]
    .
  \end{equation*}
  According to cases~(1) to ~(3) of Lemma~\ref{lem:unit-reduction} we
  deduce that $\ab[{\SL[\D/\id p^l]}]$ has order $\le 2$ or $\le 3$ in
  case (1), and has order $\le 4$ in case (2) and (3). In fact, the
  lemma implies only the bound $8$ in case (2) if $n=3$, but we know
  already from the structure of $\ab[{\SL[\Z]}]$ that $\SL[\D/\id
  p^3]\isom \SL[\Z/8\Z]$ admits exactly four characters. The existence
  of the characters $\varepsilon_\id p$, $\varepsilon_{\id p^2}$ and
  $\varepsilon_{\id p^2}'$ shows that we have indeed equality in all
  three cases, and that any linear character of $\SL[\D/\id p^l]$ is a
  power or product of these characters.

  That the subgroups occurring in~\eqref{eq:the-characters} form a
  direct product is easily seen by evaluating a product of elements of
  these groups on matrices which are congruent to the unit matrix
  modulo all the primes occurring in the products except for one, and
  using that, for $\id r^2|2$, we have $\SL[\D/\id
  r^2]\isom\SL[{\F_2[t]/(t^2)}]=\SL[\F_2]\ltimes \Gamma((t)/(t^2))$
  (see~\eqref{eq:the-irregular-decomposition}) and that
  $\varepsilon_{\id r}$ and $\varepsilon_{\id r^2}'$ vanish on one of
  these factors, respectively.  This proves Theorem~\ref{thm:main}.
\end{proof}

Let $\D$ be a field such that $\SL$ possesses a nontrivial character.
The annihilator of this character, not being equal to $\D$ since the
character is nontrivial, is the zero ideal.  Hence, by
Lemma~\ref{lem:annihilator}, $u^2=1$ for all nonzero elements in
$\D$. Since every subgroup of $\D$ is cyclic, we conclude that
$\D^\ast$ has order $1$ or $2$, i.e.~that $\D$ equals $\F_2$ or
$\F_3$.

\section{The linear characters of $\SL[\Z]$ and
  $\SL[{\F_2[t]/(t^2)}]$}
\label{sec:SL2Z}

In this section we shall prove the
formulas~\eqref{eq:explicit-formulas}
and~\eqref{eq:exceptionel-char}. Recall, first of all, that that
$\SL[\Z]$ is generated by $T=\mat11{}1$ and $S=\mat{}{-1}1{}$.
If~$\chi$ is a linear character of $\SL[\Z]$ which maps $T$ to, say
$\zeta$, then $\chi$ maps~$S$ to $\zeta^{-3}$ (since $(ST)^3=S^2$),
from which it follows that $\zeta^{12}=1$ (since $S^4=1$) and that
$\SL[\Z]$ possesses at most 12 linear characters.  In fact, since the
characters $\varepsilon_3$ and $\varepsilon_4$ define characters (as
we shall show in a moment) which have order $3$ and $4$, respectively,
we conclude that the abelianization of $\SL[\Z]$ has exact order $12$.

The formulas~\eqref{eq:explicit-formulas} result all from the
remarkable decompositions
\begin{equation}
  \label{eq:decomposition}
  G_N=\SL[\Z/N\Z]
  =
  \langle T\rangle \ltimes K_N
  \qquad
  (N=2,3,4)
  ,
\end{equation}
where $K_2$ and $K_3$ are the $3$-Sylow and $2$-Sylow subgroups of
$G_2$ and~$G_3$, respectively, and where $K_4$ is the subgroup of
$G_4$ generated by the elements of order $3$ in $G_4$. Mapping $T$ to
$-1$, $e^{2\pi i/3}$, $i$ for $N=2$, $3$, $4$, respectively defines
accordingly a character of $G_N$, which is in fact, as we shall see in
a moment, the character $\varepsilon_N$. Note that in each case $G_N$
does not possess any other character than powers of $\varepsilon_N$,
since, as we saw, the group $\SL[\Z]$ possesses only $12$
characters. In particular,~$K_N$ is the commutator subgroup of $G_N$.

Next, we prove for each $N$ the existence of the
decomposition~\eqref{eq:decomposition} and verify the claimed
formula~\eqref{eq:explicit-formulas} for $\varepsilon_N$.

The nontrivial linear character $\varepsilon_2$ of $G_2$ corresponds
to the signature map on permutations when we identify $G_2$ with the
symmetric group~$S_3$ via its natural action on the nonzero column
vectors in~$(\Z/2\Z)^3$. Thus $\varepsilon_2(A)=-1$ if $A$ corresponds
to a transposition, i.e.~if $A$ has order $2$, and
$\varepsilon_2(A)=+1$ otherwise. Note that $A$ has order~$2$ if and
only if its characteristic polynomial $x^2-tx+1$ equals $x^2-1$,
i.e.~if $t=\tr(A)=0$.  The formula~\eqref{eq:explicit-formulas} for
$\varepsilon_2$ is now obvious, as is the
decomposition~\eqref{eq:decomposition}, where $K_2$ corresponds to the
alternating subgroup under any isomorphism~$\SL[\F_2]\isom S_3$.

We note that $G_3$ has~24 elements and exactly one $2$-Sylow
subgroup~$K_3$, which is then normal.  In fact, $G_3$ has exactly $8$
elements whose order divides~$8$, which must then form the $2$-Sylow
subgroup of~$G_3$. The decomposition~\eqref{eq:decomposition} becomes
now evident.

Note that $K_3$ consists of those matrices $A$ which are either equal
to~$\pm1$ or whose trace $t$ equals~$0$. Namely, for any $A\not=\pm1$
of order dividing~$8$, its characteristic polynomial $x^2-tx+1$ must
divide the polynomial $x^8-1=(x + 1) (x - 1) (x^2 + 1) (x^2 + x - 1)
(x^2 - x - 1)$, hence equals $x^2+1$. Accordingly, we find, for any
$A=\mat *bc*$ in $G_3$, that $\varepsilon_3(A)=e^{2\pi i n/3}$, where
$n$ is chosen so that $T^{-n}A$ equals $\pm 1$ or has trace $t=0$,
i.e., where $n=-tb$ if $c=0$ or $n=tc$ otherwise. From this
formula~\eqref{eq:explicit-formulas} can be verified.

Finally, the group $G_4$ has $48$ elements. It has $8$ elements of
order~$3$. We show, first of all, that the subgroup $K_4$ generated by
these elements has order $12$. Note that this already implies the
decomposition~\eqref{eq:decomposition} since $K_4$ does then not
contain any power of $T$ except the unit matrix (since the normal
subgroup generated by $T$ or $T^2$ contains more than~$4$ elements),
and since $K_4$ is normal (by its very definition).  Clearly, $K_4$ is
contained in the inverse image $K_2'$ of the $3$-Sylow subgroup $K_2$
of~$G_2$ under the natural reduction map. But $K_2'=\langle G\rangle
\ltimes \Gamma(2)$, where $G$ is any fixed element of order $3$ (since
$K_2'$ has order $24$ and the kernel $\Gamma(2)$ of the reduction map
modulo $2$ has order $8$). The application $X\mapsto 1+2X$ defines an
isomorphism from the (additive) group $L$ of matrices over $\F_2$ with
trace~$0$ onto $\Gamma(2)$.  There is exactly one subgroup of $L$
which is invariant under conjugation with elements in~$G_2$, namely
the subgroup~$L_0$ of elements $X=\mat abc*$ in $M$ such that
$a+b+c=0$.  Accordingly, we obtain a linear character $s$ of $K_2'$ by
setting $s(A_0+2A_1)=1$ if~$A_0^{-1}A_1$ is in $L_0$, and setting
$s(A_0+2A_1)=-1$ otherwise. The kernel of $s$ contains all $8$
elements of order $3$ and has order $12$, hence coincides with the
subgroup generated by the elements of order~$3$.  We remark that $K_4$
is isomorphic to the alternating group $A_4$ (via the action of
conjugacy on its four $3$-Sylow subgroups).

Note that an element has order $3$ if and only if its trace~$t$
equals~$-1$ (since, for a matrix $A$ with trace $t$ and characteristic
polynomial $\chi_A=x^2-tx+1$, we have
$x^3-1=(x+t)\,\chi_A+(t+1)\big((t-1)x-1\big)$).  Thus, for any $A=\mat
abcd$ with trace $t$, we have $\varepsilon_4(A)=i^n$, where~$n$ is
chosen such that $T^{-n}A$ is in $K_4$.  We choose $n=(t+1)c$ if $c$
is odd (so that $T^{-n}A$ has trace $-1$), $n=b+2$ if $A=\mat **21$,
(so that $T^{-n}A=1+2\mat{}11{}$), $n=-b$ if $A=\mat **2{-1}$ (so that
$T^{-n}A=1+2\mat1{}11$), $n=b$ if if $A=\mat 1b{}1$ (so that
$T^{-n}A=1$), and $n=2-b$ if $A=\mat{-1}b{}{-1}$ (so that
$T^{-n}A=1+2\mat 11{}1$). We leave it to the reader to check
that~$i^n$ coincides in all cases with the right hand side of
formula~\eqref{eq:explicit-formulas}.

We remark that the quantities $\gs N$ occurring in
formula~\eqref{eq:explicit-formulas} depend only on the conjugacy
class of $A$. For this note that the application $A \mapsto \f =
cx^2+(d-a)xy-by^2$ maps conjugate matrices to $G_N$-equivalent
quadratic forms (more precisely, one has $f_{BAB^{-1}}(x,y) = \f
\big((x,y)B\big)$ for $B$ in $G_N$), and that $\gs N = I(f_A)$ for a
map $I$ which depends only on the $G_N$-equivalence classes of
forms. More specifically, $I(Q)=0$ if $Q=0$ or $N=4$ and $Q\equiv
0\bmod 2$, $I(Q)=1$ if $\sym{disc}(Q)\not=0$, and $I(Q)=Q(1,0)$ is a
unit and $I(Q)=-Q(0,1)$ otherwise.

Finally, we determine the linear characters of $\SL[R]$ where
$R=\F_2[t]/(t^2)$. Let $\alpha = t+(t^2)$.  Note, first of all, that
we have an isomorphism of groups $M\rightarrow \Gamma(\alpha)$,
$X\mapsto 1+\alpha X$, where $M$ is the additive subgroup of matrices
over $\F_2$ whose traces equal $0$. The group $M$ has eight
elements. Any matrix in $\SL[R]$ can be written in the form $A+\alpha
B$, where $A$ and~$B$ are matrices with entries in $\F_2$. The
application $A+\alpha B\mapsto A$ yields an exact sequence
\begin{equation*}
  1\rightarrow
  \Gamma(\alpha)
  \rightarrow \SL[R]
  \rightarrow \SL[\F_2]
  \rightarrow 1
  .
\end{equation*}
Since $\SL[\F_2]$ is a subgroup of $\SL[R]$ this sequence splits, and
we conclude that
\begin{equation}
  \label{eq:the-irregular-decomposition}
  \SL[R] = \SL[\F_2] \ltimes \Gamma(\alpha)
  .
\end{equation}
The group $\Gamma(\alpha)$ has $7$ nontrivial characters, of which
only one is invariant under conjugation with $\SL[\F_2]$, This is the
character given by $\varepsilon_4'(1+\alpha \mat abc*) =
(-1)^{a+b+c}$, which can then be continued to a character of $\SL[R]$
by setting $\varepsilon_4'(A+\alpha B) = \varepsilon_4'(1+\alpha
A^{-1}B)$. The other nontrivial character of $\SL[R]$ is
$\varepsilon_{(\alpha)}$.

\section{Statistics for number fields}
\label{sec:statistics}

We append three tables. The first one shows, for each pair of integers
$(n,r)$ ($2\le n\le 7$, $r$) the first number field~$K$ of degree $n$
and with exactly $r$ real embeddings for which $\SL [\Z_K]$ admit a
nontrivial linear character ($\Z_K$ denoting the ring of integers
of~$K$).  For generating these data we used the Bordeaux tables of
number fields~\cite{Bordeaux-tables}, which list for each pair $(n,r)$
the first few hundred thousands of number fields~$K$ having the given
signature ordered by the absolute value of their discriminant
$D_K$. The calculations were done using~\cite{sage}. According to
Corollary~\ref{cor:maximal-order-case} we had, to search for each
number fields in the Bordeaux tables for the existence of prime ideals
of degree~$1$ over $2$ or over $3$. For the totally complex fields of
degree $\ge3$, this would not suffice to make sure that there are no
additional non-congruence characters (for degree $n=2$ the first field
$\Q(\sqrt{-3})$ admits already $3$ linear characters). Here we
proceeded as follows. Once we found, for a given degree $n$, the first
field $K$ admitting a nontrivial congruence character, we checked
that, for each of the finitely many fields $L$ with $|D_L|<|D_K|$, the
ideal generated by $u^2-1$ ($u\in \Z_L^\ast$) is $1$. Since, for not
imaginary quadratic~$K$, the group $\SL[\Z_K]$ is generated by
elementary matrices~\cite{Vaserstein} and by the following proposition
we can deduce that $\SL[\Z_K]$ does not possess any nontrivial
character.
\begin{Proposition}
  Assume that the ring $\D$ is generated by elementary matrices. Then
  $\ab[{\SL[\D]}]$ is trivial if the $\D$-ideal~$\id b$ generated by
  $u^2-1$ ($u\in \D^\ast$) equals $\D$.
\end{Proposition}
\begin{proof}
  By Lemma~\ref{lem:the-key}, we have a surjective map $\D\rightarrow
  \ab[{\SL[\D]}]$, which, by Lemma~\ref{lem:annihilator}, factors
  through a surjective map
  \begin{equation}
    \label{eq:annihilator-map}
    \D/\id b\rightarrow \ab[{\SL[\D]}]
    .
  \end{equation}
  From this the proposition is obvious.
\end{proof}
Similarly, we computed the second table, which lists the first fields
$K$ of given degree~$n$ and number of real embeddings~$r$, for which
$\SL[\Z_K]$ admits the maximal possible number of linear
characters. For degree $\ge 4$ (with exception of $n=4$, $r=0$) we did
not find any such fields in the range of the Bordeaux tables.  For
$n=2$, $r=0$ we put a question mark since $\Q(\sqrt{-23})$ is the
first imaginary quadratic field such that $\SL[\Z_K]$ possesses
exactly 144 congruence characters, whereas for the order of the
abelianized group as well as for the groups $\ab[{\SL[\Z_L]}]$ for
arbitrary imaginary quadratic number fields $L$ we do not have any
information (except for some special fields $L$ which were treated in
the literature, see e.g.~\cite{Cohn}). For $n=4$, $r=0$ we computed
for all fields $L$ with $D_L<940033$ the ideals $\id b_L$ generated by
the elements $u^2-1$ ($u\in\Z_L^\ast$), and verified that $\id b_L$
decomposes into a product of pairwise different prime ideals over $3$
and squares of pairwise different prime ideals over $2$. Again from
the surjectivity of the map~\eqref{eq:annihilator-map} we can deduce
that $\SL[\Z_K]$ admits only congruence characters as linear
characters.

\begin{table}[ht]
  \label{tab:first-fields-with-chars}
  \caption{First number fields $K$ with nontrivial~$\ab[{\SL[\Z_K]}]$.}
  \footnotesize
  \begin{tabular}{|l|l|r|l|c|}
    \hline
    $n$&$r$&$D_K$&equation&$\#\ab[{\SL[\Z_K]}]$\\
    \hline
    \hline
    $2$&$0$&$-3$&$x^{2} - x + 1$&$3$\\
    $2$&$2$&$8$&$x^{2} - 2$&$4$\\
    $3$&$1$&$-31$&$x^{3} + x - 1$&$3$\\
    $3$&$3$&$81$&$x^{3} - 3x - 1$&$3$\\
    $4$&$0$&$189$&$x^{4} - x^{3} + 2x + 1$&$3$\\
    $4$&$2$&$-491$&$x^{4} - x^{3} - x^{2} + 3x - 1$&$3$\\
    $4$&$4$&$1957$&$x^{4} - 4x^{2} - x + 1$&$3$\\
    $5$&$1$&$3089$&$x^{5} - x^{3} + 2x - 1$&$3$\\
    $5$&$3$&$-9439$&$x^{5} - x^{4} - x^{3} + x^{2} - 2x + 1$&$3$\\
    $5$&$5$&$36497$&$x^{5} - 2x^{4} - 3x^{3} + 5x^{2} + x - 1$&$3$\\
    $6$&$0$&$-19683$&$x^{6} - x^{3} + 1$&$3$\\
    $6$&$2$&$63909$&$x^{6} + 2x^{4} - x - 1$&$3$\\
    $6$&$4$&$-233003$&$x^{6} - 3x^{4} - 3x^{3} + 4x^{2} + x - 1$&$3$\\
    $6$&$6$&$1259712$&$x^{6} - 6x^{4} + 9x^{2} - 3$&$3$\\
    $7$&$1$&$-435247$&$x^{7} - x^{6} - x^{5} + 3x^{4} - x^{3} - x^{2} + 2x + 1$&$3$\\
    $7$&$3$&$1602761$&$x^{7} - 2x^{6} + 2x^{5} + x^{4} - 3x^{3} + 5x^{2} - 4x + 1$&$3$\\
    $7$&$5$&$-6930439$&$x^{7} - x^{6} - 3x^{5} + 5x^{4} - 2x^{3} - 3x^{2} + 3x + 1$&$3$\\
    $7$&$7$&$25164057$&$x^{7} - 2x^{6} - 5x^{5} + 9x^{4} + 7x^{3} - 10x^{2} - 2x + 1$&$3$\\
    \hline
  \end{tabular}
\end{table}

\begin{table}[h]
  \caption{First number fields $K$ with $\#\ab[{\SL[\Z_K]}] = 12^{[K:\Q]}$}
  \footnotesize
  \begin{tabular}{|l|l|r|l|}
    \hline
    $n$&$r$&$D_K$&equation\\
    \hline
    \hline
    $2$&$0$&?\hfill$-23$&$x^{2} - x + 6$\\
    $2$&$2$&$73$&$x^{2} - x - 18$\\
    $3$&$1$&$-10079$&$x^{3} + 11x - 36$\\
    $3$&$3$&$49681$&$x^{3} - 37x - 12$\\
    $4$&$0$&$940033$&$x^{4} - x^{3} + 44x^{2} + 4x + 384$\\
    \hline
  \end{tabular}
  \label{tab:first-fields-with-max-chars}
\end{table}

\begin{table}[ht]
  \caption{First totally real abelian number fields $K$
    unramified outside a prime $l$
    with $\#\ab[{\SL[\Z_K]}] = 12^{[K:\Q]}$}
  \footnotesize
  \begin{tabular}{|r|l|}
    \hline
    $l$&equation\\
    \hline
    \hline
    $73$&$x^2 - x - 18$\\
    $307$&$x^3 - x^2 - 102 x + 216$\\
    $577$&$x^4 - x^3 - 216 x^2 + 36 x + 1296$\\
    $3221$&$x^5 - x^4 - 1288 x^3 - 17780 x^2 - 30432 x + 285696$\\
    $3889$&$x^6 - x^5 - 1620 x^4 + 360 x^3 + 174960 x^2 - 11664 x - 1259712$\\
    $5531$&$x^7 - x^6 - 2370 x^5 + 21108 x^4 + 746040 x^3 - 2927424 x^2$ \\
    &\hfill$- 46637056 x + 70303744$\\
    $6529$&$x^8 - x^7 - 2856 x^6 + 26830 x^5 + 2014493 x^4 - 23400945 x^3$\\
    &\hfill$- 374849822 x^2 + 2921535140 x + 29083370664$\\
    \hline
  \end{tabular}
  \label{tab:first-abelian-fields-with-max-chars}
\end{table}

Finally, in Table~\ref{tab:first-abelian-fields-with-max-chars} we
listed, for each degree $n\le 8$, the smallest prime $l$ such that the
$l$th cyclotomic field contains a totally real number field $K$ such
that $\SL[\Z_K]$ admits the theoretically maximal possible number of
linear characters. An equation for $K$ is given in the second column,
respectively. It is straightforward to show that as primitive element
of these fields one may take $\gamma_l=\sum_{x\bmod l} e^{2\pi i
  x^n/l}$. However, the corresponding minimal polynomial has quite
huge coefficients, so we applied to it in Sage the Pari/GP function
{\tt gp\big('polred(\dots)'\big)} to find an equation with smaller
coefficients.

In~\cite{cn-tables} the interested reader can find the results of all
our computations, in particular the order of $\ab[{\SL[\Z_K]}]$ for
all (not totally complex) number fields in the range of the Bordeaux
tables.

\bibliography{linear-characters-of-SL2O} \bibliographystyle{alpha}

\end{document}